\documentclass[reqno]{amsart}
\usepackage{graphicx} 
\usepackage[utf8]{inputenc}
\usepackage{amsmath}
\usepackage{amssymb, amsfonts}
\usepackage{amsthm,mathtools,marvosym}
\usepackage{amsmath,amsfonts,amssymb,mathrsfs,amsthm,bbm,comment}
\usepackage{tikz}
\usepackage{tikz-cd}
\usepackage{mathrsfs}
\usepackage{hyperref}
\usepackage{cleveref}
\usepackage{mathscinet}
\usepackage{placeins}
\usepackage{enumitem}

\theoremstyle{plain}
\newtheorem{theorem}{Theorem}[section]

\newtheorem{proposition}[theorem]{Proposition}

\newtheorem*{theorem*}{Theorem}

\theoremstyle{definition}

\newtheorem{example}[theorem]{Example}
\newtheorem{remark}[theorem]{Remark}

\theoremstyle{remark}

\newcount\cols
{\catcode`,=\active\catcode`|=\active
 \gdef\Young#1{\hbox{$\vcenter
 {\mathcode`,="8000\mathcode`|="8000
  \def,{\global\advance\cols by 1 &}%
  \def|{\cr
        \multispan{\the\cols}\hrulefill\cr
        &\global\cols=2 }%
  \offinterlineskip\everycr{}\tabskip=0pt
  \dimen0=\ht\strutbox \advance\dimen0 by \dp\strutbox
  \halign
   {\vrule height \ht\strutbox depth \dp\strutbox##
    &&\hbox to \dimen0{\hss$##$\hss}\vrule\cr
    \noalign{\hrule}&\global\cols=2 #1\crcr
    \multispan{\the\cols}\hrulefill\cr%
   }
 }$}}
\gdef\Skew(#1:#2){\hbox{$\vcenter
{\mathcode`,="8000\mathcode`|="8000
  \dimen0=\ht\strutbox \advance\dimen0 by \dp\strutbox
  \def\boxbeg{\vbox
    \bgroup\hrule\kern-0.4pt\hbox to\dimen0\bgroup\strut\vrule\hss$}%
  \def\boxend{$\hss\egroup\hrule\egroup}%
  \def,{\boxend\boxbeg}%
  \def|##1:{\boxend\vrule\egroup\nointerlineskip\kern-0.4pt
    \moveright##1\dimen0\hbox\bgroup\boxbeg}%
  \def\\##1\\##2:{\boxend\vrule\egroup\nointerlineskip\kern-0.4pt
    \kern ##1\dimen0\moveright##2\dimen0\hbox\bgroup\boxbeg}%
  \moveright#1\dimen0\hbox\bgroup\boxbeg#2\boxend\vrule\egroup
 }$}}
}

 \newcommand{\op}{\operatorname}

\usepackage[normalem]{ulem}
\usepackage[colorinlistoftodos]{todonotes}

\title{The branching models of Kwon and Sundaram via flagged hives}

\author[S. Kumar]{Sathish Kumar}
\address{Harish-Chandra Research Institute, Prayagraj, India}
\email{vsathishkumar@hri.res.in}

\author[J. Torres]{Jacinta Torres}
\address{Institute of Mathematics,  Jagiellonian University in Kraków, Poland}
\email{jacinta.torres@uj.edu.pl}

\date{March 2024}

\begin{document}
\begin{abstract}
We prove a bijection between the branching models of Kwon and Sundaram, conjectured previously by Lenart--Lecouvey. To do so, we use a symmetry of Littlewood--Richardson coefficients in terms of the hive model. Along the way, we obtain a new branching model in terms of \emph{flagged hives}. 
\end{abstract}
\maketitle

\section{Introduction}
In the representation theory of Lie algebras, branching problems study the restriction of a finite-dimensional irreducible highest-weight representation of a semisimple Lie algebra to nice subalgebras. Let $\mathfrak{g}$ be a semisimple Lie algebra and $\mathfrak{k}$ be a semisimple Lie subalgebra of $\mathfrak{g}$. Recall that the finite-dimensional irreducible highest-weight representations of $\mathfrak{g}$ are parametrized by dominant integral weights $\mathcal{P}^+(\mathfrak{g})$. Let $V(\nu)$ denote the irreducible highest weight representation indexed by $\nu \in \mathcal{P}^+(\mathfrak{g})$, and consider its restriction to $\mathfrak{k}$:

\begin{equation}\label{eq:branching-coeffs}
    \mathrm{res}^{\mathfrak{g}} _{\mathfrak{k}} V(\nu) = \bigoplus_{\mu \in \mathcal{P}^+ (\mathfrak k)} V(\mu) ^{\bigoplus c_{\mu} ^{\nu}}.
\end{equation}

 The coefficients $c_{\mu} ^{\nu}$ in \eqref{eq:branching-coeffs} are called branching coefficients. A combinatorial rule for computing these coefficients is called a branching rule. By a combinatorial rule we mean associating a combinatorial set (model) to each pair $(\nu, \mu)$ whose cardinality is $c^{\nu}_{\mu}$.
 \medskip

Throughout this paper, we fix $\mathfrak g$ to be the special linear Lie algebra $\mathfrak{sl}(2n, \mathbb C)$ and $\mathfrak{k}$ to be the symplectic Lie algebra $\mathfrak{sp}(2n, \mathbb C)$ (thought of as the fixed point subalgebra for the non-trivial Dynkin diagram automorphism of $\mathfrak{sl}(2n, \mathbb C)$). For the restriction problem in this case, various combinatorial models based on tableaux 
 are known. There is the classical model by Littlewood, in terms of Littlewood--Richardson tableaux, for the stable case \cite{littlewood1944invariant,littlewood1977theory}.
An elegant extension of that rule beyond the stable case was found by Sundaram \cite{sundaram1986combinatorics}. Later, a rule in terms of tableaux and Littelmann paths \cite{schumann2018non}, which was conjectured by Naito--Sagaki, was proven via its relationship to Sundaram's rule. There is another, more recent extension of the Littlewood branching rule by Kwon \cite{kwon2018combinatorial,kwon2018lusztig}, which is formulated in a more general context, for all classical types, using a combinatorial model for classical crystals known as the spinor model.
\medskip

In \cite{lecouvey2020combinatorics}, Lenart--Lecouvey use the branching models of Kwon and Sundaram to obtain combinatorial descriptions of generalized exponents in type $C_{n}$. They conjecture an explicit bijection between these two models which we intend to settle in this work. The main tool in our proof is the hive model for the Littlewood--Richardson coefficients and Gelfand-Tsetlin patterns \cite{berenstein1996canonical,henriques2006octahedron}. In fact we use the \textit{flagged hive model} studied in \cite{KRV, kundu2024saturation}, where they use the same to study saturation property for some structure constants using their connections to crystals.
\medskip

Another important element in our proof is a symmetry of Littlewood--Richardson coefficients by Kushwaha--Raghavan--Viswanath \cite{KRV}. The symmetry of Littlewood--Richardson coefficients via hive models was first studied in \cite{henriques2006octahedron} by Henriques--Kamnitzer referred to as the combinatorial $R$-matrix in \cite{lecouvey2020combinatorics}. In \cite{terada2018symmetry}, another symmetry was developed by Azenhas--King--Terada. It is an interesting question to ask whether these bijections coincide, or if they restrict to bijections between Kwon's and Sundaram's models for $c^\nu _\mu$. 

\subsection*{Organization of the paper}
 In $\S~\ref{sec:Notation}$ we fix the basic notations to work with. In $\S~\ref{sec:branching_models}$ we recall the branching models of Kwon (in fact, an equivalent formulation from \cite{lecouvey2020combinatorics}) and Sundaram in \cite{kwon2018combinatorial} and \cite{sundaram1986combinatorics}. We also state our main theorem in this section. In $\S~\ref{sec:hives}$ we recall and use the combinatorics of hives to spell out the proof.

\section{Notation}\label{sec:Notation}
A \emph{partition} is a non-increasing sequence of non-negative integers $\nu : = \nu_{1}\geq \nu_2 \geq \cdots $ such that $\nu_k =0$ for some $k \geq 1$. The maximal $j$ such that $\nu_{j} \neq 0$ is called the \emph{number of parts} or \emph{length} of $\nu$; we will denote this quantity by $\ell(\nu)$. We will abuse notation and denote a partition by $\nu = (\nu_1,..., \nu_{k})$ for $k \geq \ell(\nu)$. Given a partition $\nu$ we will often consider its Young diagram which is a left and top justified collection of boxes with $\nu_k$ many boxes in the $k^{th}$ row for all $k \in \mathbb Z _{>0}$. We will denote the Young diagram of $\nu$ again by $\nu$.  
\medskip

Let $\nu, \mu$ be partitions with $\mu \subset \nu$, (that is, the Young diagram of $\mu$ is a subset of the Young diagram of $\nu$, or equivalently, $\mu_{j} \leq \nu_{j}$ for all $ j \in \mathbb{Z}_{>0}$). A semi-standard tableau of (skew) shape $\nu/\mu$ is a function assigning a non-negative integer to each box of $\nu$ such that it is weakly increasing as we go from left to right along a row and strictly increasing as we go from top to bottom along a column, and such that it is constant and equal to $0$ precisely on the boxes corresponding to $\mu$. Usually, a positive integer $k$ will be fixed and $[0,k]:=\left\{0,1,...,k\right\}$ will be used as a co-domain for the filling function. In this case, we will denote the set of semi-standard tableaux of shape $\nu/\mu$ by $\op{SSYT}_{k}(\nu/\mu)$. The image of a box under this function will be simply referred to as the entry in the box. Whenever $\mu = (0)$, the set of semi-standard tableaux of shape $\nu/\mu$ are typically known as semi-standard Young tableaux of (straight) shape $\nu$, while semi-standard Young tableaux of shape $(m)^{d}/\nu$ are typically referred to as contretableux of shape $\nu'$, where $\nu'$ is the complement of $\nu$ in $(m)^d$ (for example, the contretableau in Example~\ref{ex:TC} has shape $(6,3,2,0)$). 
\medskip

A semistandard tableau of shape $\nu/\mu$ can equivalently be realized as a sequence of $k$ many partitions
\[\nu^{(0)} =\mu \subset \nu^{(1)} \subset \cdots \subset \nu^{(k)} = \nu,\]
\noindent
where, $\nu^{(i)}$ is defined to be the sub-shape of $\nu/\mu$ which is the pre-image of $[0,i]$. The semistandard-ness condition translates to the condition that $\nu^{(i+1)}/\nu^{(i)}$ is a horizontal strip (i.e., in any column of the Young diagram of $\nu^{(i+1)}$ there is at most one box of $\nu^{(i+1)}$ that is not a box of $\nu^{(i)}$). This is the key idea behind the definition of Gelfand-Tsetlin (GT) patterns (see \S~\ref{sec:GT patterns}) and more generally skew GT patterns (see \cite{kundu2024saturation}).
\medskip

The \textit{north-western row word}, a.k.a. the \textit{reverse row word} of a semi-standard tableau $T$ is obtained by reading the entries of its rows (excluding the entry $0$) right to left starting from the top row and proceeding downward. We will denote this word by $w(T)$ (see Example~\ref{ex:companion}). The \textit{content} of a semi-standard tableau $T$ is defined to be the content of its reverse row word. The content of a word $w$ is $\alpha := (\alpha_1, \alpha_2, \ldots)$ where, $\alpha_{i}$ equals the number of times $i$ appear in $w$. A \textit{Yamanouchi word} is a word $w = w_{1}\cdots w_{l}$ such that, for each $1\leq k \leq l$, the content of the subword  $w^{k} : = w_{1}\cdots w_{k}$ is a partition. 
\medskip

Fix a positive integer $k$. Let $L_k^*$ denote the free monoid of words in the alphabets $[1,k]$ with respect to 
the concatenation operation. The congruence relation generated by the following Knuth relations is called the plactic congruence, and the corresponding quotient monoid is called the plactic monoid:
\begin{itemize}
    \item [K1] \quad $xzy \equiv zxy$ if $x \leq y < z$
    \item [K2] \quad $yxz \equiv yzx$ if $x < y \leq z$
\end{itemize}
Two words are said to be Knuth equivalent if they are plactic congruent (i.e., they differ by a series of Knuth relations). 
\begin{theorem}\cite[Theorem~2.1]{fulton}
    Every word is Knuth equivalent to the word of a unique semistandard Young tableau.
\end{theorem}

We now briefly recall the Sch\"utzenberger involution on the set of semi-standard Young tableaux. For $T \in \mathrm{SSYT}_k(\nu)$, if $w(T) = w_1 w_2 \ldots w_m$ then the Sch\"utzenberger involution $S(T)$ is the unique tableau in the plactic class of the word $S(w(T)) := w'_m \ldots w' _2 w'_1$ where, $w'_t$ denotes $k+1-w_t$. It is a well-known fact that the tableau $S(T)$ has shape equal to $\nu$ and the map $S$ is an involution on $\mathrm{Tab}_k(\nu)$. We remark here that the word $S(w(T))$ is the reading word of a contretableau $C$ of shape $\nu'$. Thus, $S(T)$ can be defined equivalently as the rectification of this contretableau, $\mathrm{rect}(C)$. (Rectification of a semistandard tableau $T$ is the unique semistandard Young tableau whose reverse row word is Knuth equivalent to that of $T.$ Refer \cite{fulton} for more details.)

\section{The branching models of Kwon and Sundaram}\label{sec:branching_models}

Let $\nu, \mu , \lambda$ be partitions, with $\lambda, \mu \subset \nu$. From now on we will fix a positive integer $n$, and assume, unless otherwise stated, that $\ell(\nu) \leq 2n-1$. A semi-standard tableau of shape $\nu / \mu$ and content $\lambda$ is said to be a Littlewood--Richardson (LR) tableau if its reverse row word is a Yamanouchi word. We denote the set of LR tableaux of shape $\nu/ \mu$ and content $\lambda$ by $\operatorname{LR}(\nu/\mu, \lambda)$. The numbers $c^{\nu}_{\mu,\lambda} := |\operatorname{LR}(\nu/\mu, \lambda)|$ are called the \emph{LR coefficients}.  Let $T_{\mu}$ be the unique semi-standard Young tableau of shape and content $\mu$.  We say that a semi-standard Young tableau $T$ is $\mu$-\emph{dominant} if it satisfies the following condition: 

\begin{align}
\label{companionprop}
w(T_{\mu})*w(T) \hbox{ is a Yamanouchi word}. 
\end{align}
We denote the set of all $\mu$-dominant semi-standard Young tableau of shape $\lambda$ and content $\nu - \mu$ by $\operatorname{LR^{\nu}_{\mu,\lambda}}$.

\begin{example}
\label{lrex}
Let $\nu = (5,3,1), \mu = (3,1), \lambda = (3,1,1)$. In this case, there is a unique Littlewood--Richardson tableau of shape $\nu/ \mu$ and content $\lambda$: 
\[T = \Skew(0: 0,0,0, 1,1|0:0, 1,2 |0: 3).\]
\end{example}

Now, from a Littlewood--Richardson tableau $T$ of shape $\nu/ \mu$ and content $\lambda$ one can easily obtain its \emph{companion tableau} $c(T)$ by placing in the $k$-th row of the Young diagram of $\lambda$ the numbers of the rows of $T$ containing an
entry $k$. 

\begin{example}
\label{ex:companion}
For $T$ as in Example \ref{lrex}, we have 
\[c(T) = \Skew(0:1,1,2|0: 2|0:3) \hbox{ and also }  T_{\mu} = \Skew(0:1,1,1|0:2), w(T_{\mu}) = 1112, w(T) = 21123.\]

Note that $w(T_{\mu})*w(T) =111221123 $ is a Yamanouchi word. 

\end{example}

The companion tableau $c(T)$ of $T$ is a $\mu$-dominant semi-standard Young tableau of shape $\nu$ and weight $\nu - \mu$. 
In fact, it is well known that the companion map induces a bijection 
\[ c: \operatorname{LR}(\nu/\mu, \lambda) \ \tilde{\longrightarrow} \ \operatorname{LR}^{\nu}_{\mu,\lambda}. \]
\subsection{Representation theory  and symmetry of LR--coefficients}
 
The set of dominant integral weights for the special linear Lie algebra $\mathfrak{sl}(2n,\mathbb{C})$ is in one-to-one correspondence with the set of integer partitions $\nu$ for which $\ell(\nu)\leq 2n-1$. Therefore the irreducible finite-dimensional highest weight representations of $\mathfrak{sl}(2n,\mathbb{C})$ are indexed by partitions $\nu$ for which $\ell(\nu) \leq 2n-1$. Let $V(\nu)$ be the finite-dimensional irreducible $\mathfrak{sl}(2n,\mathbb{C})$ module indexed by $\nu$. For partitions $\lambda$ and $\mu$ whose length is at most $2n-1$, the Littlewood--Richardson coefficients are the tensor product multiplicities defining the branching 

\[V(\lambda)\otimes V(\mu) \cong \bigoplus_{\nu} V(\nu)^{\bigoplus c^{\nu}_{\mu,\lambda}}.\]

By the symmetry of tensor products, it is then clear that $c^{\nu}_{\mu,\lambda} = c^{\nu}_{\lambda,\mu}$; this property is called the \emph{symmetry of Littlewood--Richardson coefficients}. In this work, we will recurr to a bijection $\operatorname{LR}^{\nu}_{\mu,\lambda} \overset{U}{\longleftrightarrow} \operatorname{LR}^{\nu}_{\lambda,\mu}$ from \cite{KRV} via the \emph{hive model} (cf. \S~\ref{sec:hives}). 

\subsection{Branching models}
The set of irreducible finite-dimensional highest weight representations for the symplectic Lie algebra $\mathfrak{sp}(2n,\mathbb{C})$ are indexed by partitions $\mu$ for which $\ell(\mu) \leq n$. Let $V^{\sigma}(\mu)$ denote the simple $\mathfrak{sp}(2n, \mathbb C)$ module of highest weight $\mu$. From now on, in this section, we fix $\nu$ to be a partition with at most $2n-1$ parts and $\mu$ a partition with at most $n$ parts. Consider the branching of $V(\nu)$ after restriction to $\mathfrak{sp}(2n,\mathbb{C})$:

\[\operatorname{res}^{\mathfrak{sl}(2n,\mathbb{C})}_{\mathfrak{sp}(2n,\mathbb{C})} V(\nu)=  \bigoplus_{\mu} V^{\sigma}(\mu) ^{\bigoplus c^\nu _\mu}.\]

A partition $\lambda$ is \emph{even} if $\lambda_{2i-1} = \lambda_{2i}$ for each $i \in \mathbb{Z}_{> 0}$. Let $\lambda$ be an even partition.

\subsubsection{Sundaram's branching model} We say that a Littlewood--Richardson tableau of shape $\nu/ \mu$ and content $\lambda$ satisfies the \textit{Sundaram} property, if, for each $i = 0,..., \frac{1}{2}\ell(\lambda)$, the entry $2i+1$ appears in row $n+i$ or above in the Young diagram of $\nu$. We denote the set of $ T \in \operatorname{LR}(\nu/\mu, \lambda)$ satisfying the Sundaram property by $\operatorname{LRS}(\nu/\mu, \lambda)$. 

\begin{example}
\label{ex:lrs}
Let $n = 3$. The tableau in Example \ref{lrex} satisfies the Sundaram condition, but 
\[\Skew(0: 0, 0, 0, 1,1|0:0, 1,2 |0: 0,3|0:1).\]
does not. 
\end{example}

The following theorem is due to Sundaram.

\begin{theorem}\cite[Theorem 12.1]
{sundaram1986combinatorics}
\label{sundaram}
The branching coefficient $c^\nu_\mu$ equals the cardinality of the set 
\[\operatorname{LRS}(\nu,\mu): = \bigcup \operatorname{LRS}(\nu/\mu, \lambda),\]
where the union is taken over all even partitions $\lambda$. 

\end{theorem}

\subsubsection{Kwon's branching model}
A tableau of shape $\mu$ (recall that $\ell(\mu) \leq n$) is said to satisfy the \emph{Kwon property} if the entries in row $i$ are at least $2i-1$, for $i = 1,...,n$. Denote the subset of $\operatorname{LR}^{\nu}_{\lambda,\mu}$ consisting of tableaux $T$ such that their evacuation $S(T)$ satisfies the Kwon property by $LRK_{\lambda, \mu}^{\nu}$. 
\medskip

The following theorem is a reformulation of Kwon's branching rule by Lecouvey--Lenart \cite[Lemma 6.11]{lecouvey2020combinatorics}.
\begin{theorem}\cite{kwon2018combinatorial,lecouvey2020combinatorics}
\label{kwon}
    The branching coefficient $c^\nu_\mu$ equals the cardinality of  the set 

\[\operatorname{LRK}(\nu,\mu): = \bigcup LRK_{\lambda, \mu}^{\nu},\]

\noindent where, the union is taken over all even partitions $\lambda$. 
\end{theorem}

\subsection{Main result}
We state our main theorem below. 
\begin{theorem}
\label{maintheorem}
The composition
\[\operatorname{LR}(\nu/\mu,\lambda) \overset{c}{\longrightarrow} \operatorname{LR}^{\nu}_{\mu,\lambda} \overset{U}{\longrightarrow} \operatorname{LR}^{\nu}_{\lambda,\mu}\]

\noindent restricts to a bijection 

\[\operatorname{LRS}(\nu/\mu,\lambda)\tilde{\rightarrow} \operatorname{LRK}^{\nu}_{\lambda,\mu},\]

\noindent where $U: \operatorname{LR}^{\nu}_{\mu,\lambda} \tilde{\rightarrow} \operatorname{LR}^{\nu}_{\lambda,\mu}$ is the bijection by Kushwaha--Raghavan--Viswanath \cite{KRV}. (In fact, $U = \mathrm{rect}\circ C \circ \hat{P} \circ \varphi$; see Section \ref{sec:hives} for notation). Therefore, the above composition induces a bijection between $\operatorname{LRS}(\nu,\mu)$ and $\operatorname{LRK}(\nu,\mu)$.
\end{theorem}
In \cite{lecouvey2020combinatorics}, Lenart--Lecouvey conjectured a very similar bijection, induced by the composition

\[\operatorname{LR}(\nu/\mu,\lambda) \overset{c}{\longrightarrow} \operatorname{LR}^{\nu}_{\mu,\lambda} \overset{U'}{\longrightarrow} \operatorname{LR}^{\nu}_{\lambda,\mu},\]

where $U'$ is the symmetry defined by Henriques--Kamnitzer, also known as the combinatorial $R$-matrix.

We prove Theorem \ref{maintheorem} in Section \ref{sec:hives} and as a byproduct, we obtain a new branching model in terms of the \emph{flagged hive model} (see Remark~\ref{rem-hive-rest}).

\section{The flagged hive model}
\label{sec:hives}
In this section, we recollect the notions of Gelfand-Tsetlin patterns, hives and their connections with tableaux and use it to prove Theorem~\ref{maintheorem}.

\subsection{Gelfand-Tsetlin Patterns}\label{sec:GT patterns}
Fix a positive integer $m$. A Gelfand-Tsetlin (GT) pattern $P$ is a triangular array of numbers $(p_{i,j})_{\begin{smallmatrix}1\leq i \leq m,\\ 1 \leq j \leq i\end{smallmatrix}}$ such that 
\begin{equation} \label{eq_GT_inequality}
p_{i+1,j} \geq p_{ij} \geq p_{i+1,j+1}
\end{equation}
for all appropriate values of $i$ and $j$. 
\medskip

\begin{figure}[h]
    \centering
    \begin{tikzpicture}
    \node at (0.5,0.5*1.732) {$p_{4,1}$};
    \node at (1.5,0.5*1.732) {$p_{4,2}$};
    \node at (2.5,0.5*1.732) {$p_{4,3}$};
    \node at (3.5,0.5*1.732) {$p_{4,4}$};

    \node at (1,1.732) {$p_{3,1}$};
    \node at (2,1.732) {$p_{3,2}$};
    \node at (3,1.732) {$p_{3,3}$};

    \node at (1.5,1.5*1.732) {$p_{2,1}$};
    \node at (2.5,1.5*1.732) {$p_{2,2}$};
        
    \node at (2,2*1.732) {$p_{1,1}$};
\end{tikzpicture}
\begin{tikzpicture}
    \node at (-1,0.5*1.732) {$\null$};
    \node at (0.5,0.5*1.732) {$6$};
    \node at (1.5,0.5*1.732) {$3$};
    \node at (2.5,0.5*1.732) {$2$};
    \node at (3.5,0.5*1.732) {$0$};

    \node at (1,1.732) {$4$};
    \node at (2,1.732) {$3$};
    \node at (3,1.732) {$0$};

    \node at (1.5,1.5*1.732) {$3$};
    \node at (2.5,1.5*1.732) {$1$};
        
    \node at (2,2*1.732) {$1$};
\end{tikzpicture}
    \caption{Gelfand-Tsetlin patterns}
    \label{fig-GT}
\end{figure}

The inequalities in \eqref{eq_GT_inequality} implies that $P_k := (p_{k,1}, p_{k,2}, \ldots, p_{k,k},0,0, \ldots, 0)$ is weakly decreasing ($1 \leq k \leq m$). This in turn implies that an integral GT pattern $P$ (i.e., $p_{i,j} \in \mathbb{Z}_{\geq 0}$) is in fact a sequence of partitions $(0) \subset P_1 \subset P_2 \subset \ldots \subset P_m$ such that the length of $P_k$ is at most $k$ and the successive quotients $P_{k+1} / P_{k}$ are horizontal strips ($1 \leq k \leq m-1$). Recall from Section \ref{sec:Notation} that this defines a semi-standard Young tableau of shape $P_m$. We denote this semi-standard tableau by $T(P)$. Note that this map defines a bijection between GT patterns and semi-standard Young tableaux. Given a semi-standard Young tableau $R$, we will denote the associated GT pattern by $GT(R)$. 
\medskip

Given an integral GT pattern $P$, one could also define a contretableau $C(P)$ as follows:\\ First, let $k $ denote the largest part of $P_m$ (i.e., $k = p_{m,1}$). Then the sequence of partitions $\textbf{k}-\mathrm{rev}(P_m) \subset \textbf{k}-\mathrm{rev}(P_{m-1}) \subset \cdots \subset \textbf{k}-\mathrm{rev}(P_{1}) \subset \textbf{k}$ defines a contretableau, where $\mathrm{rev}(\gamma)$ denotes $(\gamma_m, \ldots, \gamma_1)$, the reverse of $\gamma = (\gamma_1, \ldots, \gamma_m)$ and $\textbf{k}$ denotes the partition $(k, \ldots, k)$ with $m$ parts. This is because, if $\lambda$ and $\mu$ are partitions with $\mu \subset \lambda$ then and if $k \geq \lambda_1$ then $\textbf{k}-\mathrm{rev}(\lambda) \subset \textbf{k}-\mathrm{rev}(\mu)$; moreover, it is easy to see with help of Young diagrams that if $\lambda/\mu$ is a horizontal strip, then so is $\textbf{k}-\mathrm{rev}(\mu) / \textbf{k}-\mathrm{rev}(\lambda)$.

\begin{example}
    \label{ex:TC}
Let $P$ be the GT pattern appearing in Figure~\ref{fig-GT}. We have $m = 4$, and $P_4 = (6,3,2,0), P_3 = (4,3,0,0), P_2 = (3,1,0,0), P_1 = (1,0,0,0)$. Then 

\begin{align*}
T(P) = \Skew(0:1,2,2,3,4,4|0:2,3,3|0:4,4) \hbox{ and } C(P) = \Skew(0:0,0,0,0,0,0|0:0,0,0,0,1,1|0:0,0,0,2,2,3|0:1,1,2,3,3,4).
\end{align*}

Recall the Sch\"utzenberger involution S on semi-standard Young tableaux defined in Section \ref{sec:Notation}. It is well-known that this operation coincides with the Lusztig--Sch\"utzenberger involution in the context of crystals \cite{schutzenberger1972promotion, gansner1980equality, berenstein1996canonical}. Note that the contents of $T(P)$ and $C(P)$ are reverses of each other. In fact, they are related by the Sch\"utzenberger involution. This is made precise in Proposition \ref{prop_SI} below.

\end{example}
\begin{proposition}\label{prop_SI}
    Given a GT pattern $P$, the tableau and contretableau associated with it are swapped by the Sch\"utzenberger involution, that is $S(T(P)) = \mathrm{rect}(C(P))$.
\end{proposition}

\begin{proof}
    Let $P$ be a GT pattern of size $m$. It follows from the definitions of $T(P)$ and $C(P)$ that the reverse reading word of $T(P)$, respectively the reverse reading word of $C(P)$, are obtained from the NE diagonals of $P$ in the following way. The $i$-th NE diagonal of $P$, $p_{m,i},..., p_{i,i}$ determines the entries in the i-th row of $T(P)$ read from right to left, respectively the entries in the $m-i$-th row of $C(P)$ read from left to right. This is done very simply: the number of $j's$ in row $i$ of $T(P)$, respectively the number of $m-j's$ in row $m-i$ of $C(P)$, is given by $p_{j,i} - p_{j-1,i}$, working with the convention that $p_{i,j} = 0$ whenever $j <i$. 
    
\end{proof}

\subsection{Hive polytopes}
Fix a positive integer $m$. A $m+1$-triangular grid as in Figure~\ref{4-hive-grid} will be called a $m$-hive triangle. 
\begin{figure}[h]
    \centering
    \begin{center}
    \begin{tikzpicture}
        \node at (0,0) {$\bullet$};
        \node at (1,0) {$\bullet$};
        \node at (2,0) {$\bullet$};
        \node at (3,0) {$\bullet$};
		\node at (4,0) {$\bullet$};
        
        \node at (0.5,0.5*1.732) {$\bullet$};
        \node at (1.5,0.5*1.732) {$\bullet$};
        \node at (2.5,0.5*1.732) {$\bullet$};
        \node at (3.5,0.5*1.732) {$\bullet$};

        \node at (1,1.732) {$\bullet$};
        \node at (2,1.732) {$\bullet$};
        \node at (3,1.732) {$\bullet$};

        \node at (1.5,1.5*1.732) {$\bullet$};
        \node at (2.5,1.5*1.732) {$\bullet$};
        
        \node at (2,2*1.732) {$\bullet$};

        %horizontal lines
        \draw (0,0) -- (4,0);
		\draw (0.5, 0.5*1.732) -- (3.5, 0.5*1.732);
		\draw (1, 1.732) -- (3, 1.732);
		\draw (1.5, 1.5*1.732) -- (2.5, 1.5*1.732);

        %NE lines
		\draw (0,0) -- (2, 2*1.732);
		\draw (1,0) -- (2.5, 1.5*1.732);
		\draw (2,0) -- (3, 1.732);
		\draw (3,0) -- (3.5, 0.5* 1.732);	

        %NW lines
		\draw (1,0) -- (0.5, 0.5*1.732);
		\draw (2,0) -- (1,1.732);
		\draw (3,0) -- (1.5, 1.5*1.732);
        \draw (4,0) -- (2, 2*1.732);
	\end{tikzpicture}
 \end{center}
    \caption{The 4-hive triangle}
    \label{4-hive-grid}
\end{figure}
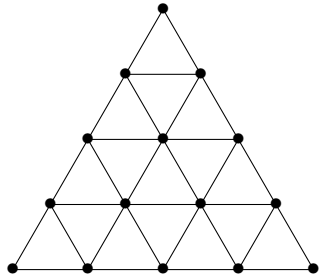
\medskip

Observe that the unit rhombi in a $m$-hive triangle are of three kinds based on their orientation:
$$\text{Northeast (NE):} \;\;\; 
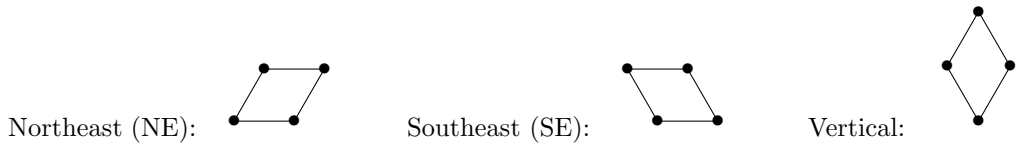
\begin{tikzpicture}
  [scale=0.4]
  \node at (0,0) {$\bullet$};
  \node at (2,0) {$\bullet$};
  \node at (1,1.732) {$\bullet$};
  \node at (3,1*1.732) {$\bullet$};
  \draw (0,0) -- (2,0);
  \draw (0,0) -- (1, 1.732);
  \draw (1, 1.732) -- (3, 1.732);
  \draw (2,0) -- (3,1.732);
  
\end{tikzpicture}\; \quad \quad \;
 \text{Southeast (SE):} \;\;\;   \begin{tikzpicture}
  [scale=0.4]
  \node at (0,0) {$\bullet$};
  \node at (2,0) {$\bullet$};
  \node at (-1,1.732) {$\bullet$};
  \node at (1,1.732) {$\bullet$};
  \draw (0,0) -- (2,0);
  \draw (0,0) -- (-1, 1.732);
  \draw (-1, 1.732) -- (1, 1.732);
  \draw (2,0) -- (1,1.732);
  
\end{tikzpicture}\;    \quad \quad \;
\; \text{Vertical: } \;\;
\begin{tikzpicture}
 [scale=0.3]
  \node at (0,0) {$\bullet$};
  \node at (1.4,1.4*1.732) {$\bullet$};
  \node at (1.4,-1.4*1.732) {$\bullet$};
  \node at (2.8,0) {$\bullet$};
  \draw (0,0) -- (1.4,1.4*1.732);
  \draw (2.8,0) -- (1.4,1.4*1.732);
  \draw (1.4,-1.4*1.732) -- (0,0);
  \draw (1.4,-1.4*1.732) -- (2.8,0);
  
\end{tikzpicture}$$

A $m$-hive is a labelling of the $m+1$-hive triangle such that the content of each small rhombus is positive. Here, the \emph{content} of a small rhombus is the sum of the labels on its obtuse-angled nodes minus the sum of
the labels on its acute-angled nodes (in Figure~\ref{fig:one_flat_rhombi}, the content of the displayed NE rhombus would be $h_{i,j}h_{i+1,j+1}-h_{i+1,j}h_{i,j+1}$). 
\medskip

Given partitions $\lambda, \mu$ and $\nu$ with at most $m$ non-zero parts, the hive polytope $\mathrm{Hive(\lambda, \mu, \nu)}$ is the set of all $m$-hives with the boundaries labelled as in Figure~\ref{hive_boundary} (left).
\begin{figure}
\centering
\begin{center}\null
\begin{tikzpicture}
    \node at (0,0) {$\bullet$};
    \node at (1,0) {$\bullet$};
    \node at (2,0) {$\bullet$};
    \node at (3,0) {$\bullet$};
	\node at (4,0) {$\bullet$};
    
    \node at (-0.3,-0.5) {$|\lambda|$};
    \node at (1,-0.5) {$|\lambda|$};
    \node at (1,-1) {$+ \mu_1$};
    \node at (2,-0.5) {$|\lambda|$};
    \node at (2,-1) {$+ \mu_1 $};
    \node at (2,-1.5) {$+ \mu_2$};
    \node at (3,-0.5) {$\cdots$};
    \node at (5,-0.5) {$|\lambda| + |\mu| = |\nu|$};
        
    \node at (0.5,0.5*1.732) {$\bullet$};
    \node at (1.5,0.5*1.732) {$\bullet$};
    \node at (2.5,0.5*1.732) {$\bullet$};
    \node at (3.5,0.5*1.732) {$\bullet$};

    \node at (-0.1,0.58*1.732) {$.$};
    \node at (-0.17,0.5*1.732) {$.$};
    \node at (-0.03,0.66*1.732) {$.$};
    \node at (0.1,1.732+0.1) {$\lambda_1 + \lambda_2$};
    \node at (1,1.5*1.732+0.1) {$\lambda_1$};
    \node at (1.7,2*1.732+0.2) {$0$};

    \node at (3,1.5*1.732) {$\nu_1$};
    \node at (3.8,1.732+0.1) {$\nu_1 + \nu_2$};
    \node at (4.2, 0.5*1.732+0.2) {$.$};
    \node at (4.27, 0.42*1.732+0.2) {$.$};
    \node at (4.34, 0.34*1.732+0.2) {$.$};

    \node at (1,1.732) {$\bullet$};
    \node at (2,1.732) {$\bullet$};
    \node at (3,1.732) {$\bullet$};

    \node at (1.5,1.5*1.732) {$\bullet$};
    \node at (2.5,1.5*1.732) {$\bullet$};
        
    \node at (2,2*1.732) {$\bullet$};
\end{tikzpicture}
\begin{tikzpicture}
    \node at (-1.7,-1.7) {$\null$};
    \node at (0,0) {$\bullet$};
    \node at (1,0) {$\bullet$};
    \node at (2,0) {$\bullet$};
    \node at (3,0) {$\bullet$};
	\node at (4,0) {$\bullet$};
    \node at (0,-0.3) {$h_{5,1}$};
    \node at (1,-0.3) {$h_{5,2}$};
    \node at (2,-0.3) {$h_{5,3}$};
    \node at (3,-0.3) {$h_{5,4}$};
	\node at (4,-0.3) {$h_{5,5}$};
  
    \node at (0.5,0.5*1.732) {$\bullet$};
    \node at (1.5,0.5*1.732) {$\bullet$};
    \node at (2.5,0.5*1.732) {$\bullet$};
    \node at (3.5,0.5*1.732) {$\bullet$};
        
    \node at (0.5,0.3*1.732) {$h_{4,1}$};
    \node at (1.5,0.3*1.732) {$h_{4,2}$};
    \node at (2.5,0.3*1.732) {$h_{4,3}$};
    \node at (3.5,0.3*1.732) {$h_{4,4}$};

    \node at (1,0.8*1.732) {$h_{3,1}$};
    \node at (2,0.8*1.732) {$h_{3,2}$};
    \node at (3,0.8*1.732) {$h_{3,3}$};
    
    \node at (1.5,1.3*1.732) {$h_{2,1}$};
    \node at (2.5, 1.3*1.732) {$h_{2,2}$};

    \node at (2,1.8*1.732) {$h_{1,1}$};
        
    \node at (1,1.732) {$\bullet$};
    \node at (2,1.732) {$\bullet$};
    \node at (3,1.732) {$\bullet$};

    \node at (1.5,1.5*1.732) {$\bullet$};
    \node at (2.5,1.5*1.732) {$\bullet$};
        
    \node at (2,2*1.732) {$\bullet$};
\end{tikzpicture}
\end{center}
    \caption{}
    \label{hive_boundary}
\end{figure}
\subsubsection{Hives and LR coefficients}
\begin{theorem}\cite{Buch,Knutson-Tao}
    The LR coefficient $c_{\lambda, \mu} ^{\nu}$ is given by the number of integral points in the hive polytope $\mathrm{Hive}(\lambda, \mu, \nu)$.
\end{theorem}

We present here the bijective map $\varphi: \operatorname{LR}^{\nu}_{\lambda,\mu} \longrightarrow  \mathrm{Hive(\lambda, \mu, \nu)}$ for the comfort of the reader. Let $R \in \operatorname{LR}^{\nu}_{\lambda,\mu}$. We first compute the GT pattern $GT(R)$ defined by $R$. Next, we obtain a new $m+1$ triangular arrangement of numbers $$GT(R)^{p} := (a_{i,j})_{\begin{smallmatrix}0\leq i \leq m,\\ 0 \leq j \leq i\end{smallmatrix}}$$ defined by taking partial sums along the rows of $GT(R)$ i.e., define $a_{0,0} := 0$ and for $1 \leq i \leq m$, define $a_{i,j} : = \sum_{k=1}^{j} p_{i,k}$, where $ p_{i,j} $ are the entries of the GT pattern $GT(R)$. Now let $\lambda^{p} = (\lambda^{p}_{0}, \lambda^{p}_{1},..., \lambda^{p}_{m})$ denote the $m+1$ vector consisting of partial sums for $\lambda$, i.e., $\lambda^{p}_{j} = \sum_{i = 1}^{j} \lambda^{p}_{i}$. The hive $\varphi(R)$ is obtained from $GT(R)^{p}$ by adding  $\lambda^{p}_{j}$ to each entry of $GT(R)^{p}$ that has the form $a_{i,j}$. That is,
$$\varphi(R) := (a_{i,j} + \lambda_j ^p)_{\begin{smallmatrix}0\leq i \leq m,\\ 0 \leq j \leq i\end{smallmatrix}} $$
\begin{example}
    \label{ex:lr-comp-gt-h}
Let $m = 6, \nu = (5,4,3,3,0,0), \lambda = (2,1,1,0,0,0)$ and $\mu = (4,3,2,2,0,0)$, and let 
\[ T = \Skew(0: 0, 0,1,1,1|0: 0,2,2,2 |0: 0, 3,3 |0: 1,4,4) \in \operatorname{LR}(\nu/ \lambda, \mu), \hbox{ whose companion tableau is }
 c(T) = \Skew(0: 1,1,1,4 |0: 2,2,2|0: 3,3 |0: 4,4).\]

Then we have $GT(T)$ and $GT(T)^p$ are given by 

\begin{figure}[h]
\centering
\begin{tikzpicture}
    \node at (-1,0.5*1.732) {$\null$};

    \node at (-0.5,-0.5*1.732) {$4$};
    \node at (0.5,-0.5*1.732) {$3$};
    \node at (1.5,-0.5*1.732) {$2$};
    \node at (2.5,-0.5*1.732) {$2$};
    \node at (3.5,-0.5*1.732) {$0$};
    \node at (4.5,-0.5*1.732) {$0$};

    \node at (0,0) {$4$};
    \node at (1,0) {$3$};
    \node at (2,0) {$2$};
    \node at (3,0) {$2$};
    \node at (4,0) {$0$};
    
    \node at (0.5,0.5*1.732) {$4$};
    \node at (1.5,0.5*1.732) {$3$};
    \node at (2.5,0.5*1.732) {$2$};
    \node at (3.5,0.5*1.732) {$2$};

    \node at (1,1.732) {$3$};
    \node at (2,1.732) {$3$};
    \node at (3,1.732) {$2$ };

    \node at (1.5,1.5*1.732) {$3$};
    \node at (2.5,1.5*1.732) {$3$};
        
    \node at (2,2*1.732) {$3$};

\end{tikzpicture}~\begin{tikzpicture}
    \node at (-1,0.5*1.732) {$\null$};
    \node at (0.5,-0.5*1.732) {$0$};
    \node at (1.5,-0.5*1.732) {$4$};
    \node at (2.5,-0.5*1.732) {$7$};
    \node at (3.5,-0.5*1.732) {$9$};
    \node at (4.5,-0.5*1.732) {$11$};
    \node at (5.5,-0.5*1.732) {$11$};
    \node at (6.5,-0.5*1.732) {$11$};

    \node at (1,0) {$0$};
    \node at (2,0) {$4$};
    \node at (3,0) {$7$};
    \node at (4,0) {$9$};
    \node at (5,0) {$11$};
    \node at (6,0) {$11$};
    
    \node at (1.5,0.5*1.732) {$0$};
    \node at (2.5,0.5*1.732) {$4$};
    \node at (3.5,0.5*1.732) {$7$};
    \node at (4.5,0.5*1.732) {$9$};
    \node at (5.5,0.5*1.732) {$11$.};
    
    \node at (2,1.732) {$0$};
    \node at (3,1.732) {$3$};
    \node at (4,1.732) {$6$};
    \node at (5,1.732) {$8$};

    \node at (2.5,1.5*1.732) {$0$};
    \node at (3.5,1.5*1.732) {$3$};
    \node at (4.5,1.5*1.732) {$6$ };

    \node at (3,2*1.732) {$0$};
    \node at (4,2*1.732) {$3$};
        
    \node at (3.5,2.5*1.732) {$0$};
\end{tikzpicture}
\end{figure}

\noindent
The partial sums vector for $\lambda$ is $(0,2,3,4,4,4,4)$. We therefore proceed to add the entries of the partial sums vector to  $GT(T)^p$ to get the hive $\varphi(c(T))$, that is, we add $0$ to the first row, $2$ to the second row, $3$ to the third row, and $4$ to the 5-th, 6-th and 7-th rows: 

\begin{figure}[h]
    \centering
    \begin{tikzpicture}
    \node at (5,-1.732) {$4$};
    \node at (6,-1.732) {$8$};
    \node at (7,-1.732) {$11$};
    \node at (8,-1.732) {$13$};
    \node at (9,-1.732) {$15$};
    \node at (10,-1.732) {$15$};
    \node at (11,-1.732) {$15$};

    \node at (5.5,-0.5*1.732) {$4$};
    \node at (6.5,-0.5*1.732) {$8$};
    \node at (7.5,-0.5*1.732) {$11$};
    \node at (8.5,-0.5*1.732) {$13$};
    \node at (9.5,-0.5*1.732) {$15$};
    \node at (10.5, -0.5*1.732) {$15$};

    \node at (6,0) {$4$};
    \node at (7,0) {$8$};
    \node at (8,0) {$11$};
    \node at (9,0) {$13$};
    \node at (10,0) {$15$};
    
    \node at (6.5,0.5*1.732) {$4$};
    \node at (7.5,0.5*1.732) {$7$};
    \node at (8.5,0.5*1.732) {$10$};
    \node at (9.5,0.5*1.732) {$12$};

    \node at (7,1.732) {$3$};
    \node at (8,1.732) {$6$};
    \node at (9,1.732) {$9$ };

    \node at (7.5,1.5*1.732) {$2$};
    \node at (8.5,1.5*1.732) {$5$};

    \node at (8,2*1.732) {$0$};
\end{tikzpicture}
\end{figure}
\end{example}
\subsubsection{Flagged hives}
A flag $\phi = (\phi_1, \ldots, \phi_m)$ is a weakly increasing $m$-tuple of positive integers such that $i \leq \phi_i \leq m$ for all $1 \leq i \leq m$. The \emph{flagged hive polytope}, denoted by  $\mathrm{Hive}(\lambda, \mu, \nu, \phi)$, corresponding to a flag $\phi$ is defined to be the set of all hives in $\mathrm{Hive}(\lambda, \mu, \nu)$ for which given any $k$, the contents of the first $m-\phi_k$ northeast rhombi in the $k^{th}$ (slanted) column are $0$.
\medskip

See Figure~\ref{flagged_hive} for an illustration of the set of all northeast rhombi (shaded in the figure) determined by the flag $(2,3,3,4)$ whose contents are all required to be $0$. 

\begin{figure}[h]
    \centering
    \begin{center}
    \begin{tikzpicture}
        \node at (0,0) {$\bullet$};
        \node at (1,0) {$\bullet$};
        \node at (2,0) {$\bullet$};
        \node at (3,0) {$\bullet$};
		\node at (4,0) {$\bullet$};
        
        \node at (0.5,0.5*1.732) {$\bullet$};
        \node at (1.5,0.5*1.732) {$\bullet$};
        \node at (2.5,0.5*1.732) {$\bullet$};
        \node at (3.5,0.5*1.732) {$\bullet$};

        \node at (1,1.732) {$\bullet$};
        \node at (2,1.732) {$\bullet$};
        \node at (3,1.732) {$\bullet$};

        \node at (1.5,1.5*1.732) {$\bullet$};
        \node at (2.5,1.5*1.732) {$\bullet$};
        
        \node at (2,2*1.732) {$\bullet$};

        %horizontal lines
        \draw (0,0) -- (4,0);
		\draw (0.5, 0.5*1.732) -- (3.5, 0.5*1.732);
		\draw (1, 1.732) -- (3, 1.732);
		\draw (1.5, 1.5*1.732) -- (2.5, 1.5*1.732);

        %NE lines
		\draw (0,0) -- (2, 2*1.732);
		\draw (1,0) -- (2.5, 1.5*1.732);
		\draw (2,0) -- (3, 1.732);
		\draw (3,0) -- (3.5, 0.5* 1.732);	

        %NW lines
		\draw (1,0) -- (0.5, 0.5*1.732);
		\draw (2,0) -- (1,1.732);
		\draw (3,0) -- (1.5, 1.5*1.732);
        \draw (4,0) -- (2, 2*1.732);

        \draw [ultra thick, draw=black, fill=violet, fill opacity=0.5]
       (0,0) -- (1,1.732) -- (2,1.732) -- (1.5,0.5*1.732) -- (3.5,0.5*1.732) --  
       (3,0) -- cycle;

       \draw [ultra thick, draw=black]
       (0.5,0.5*1.732) -- (1.5,0.5*1.732) -- (1,0);
        \draw [ultra thick, draw=black]
       (2.5,0.5*1.732) -- (2,0);
        
	\end{tikzpicture}
 \end{center}
    \caption{The region corresponding to the flag $\phi = (2,3,3,4)$ for a $4-$hive.}

    \label{flagged_hive}
\end{figure}
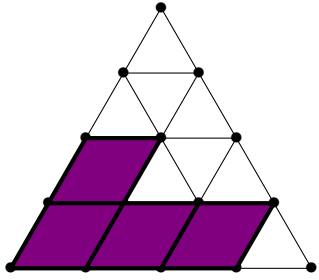

Given a partition $\lambda$ and a flag $\phi$, we define the set of flagged tableaux $\mathrm{SSYT}(\lambda, \phi)$ to be the set of all semi-standard Young tableaux of shape $\lambda$ such that each entry in the $k^{th}$ row is at most $\phi_k$, for all $k$. 
\medskip

\begin{proposition}\label{prop_fl_hive_to_tab}\cite{KRV} 
    For partitions $\lambda, \mu, \nu$ with at most $m$ parts, the set of $\mu$-dominant flagged tableaux of shape $\lambda$ with flag $\phi$ and weight $\nu - \mu$ (i.e., $\mathrm{LR}^\nu _{\mu,\lambda} \cap \mathrm{SSYT}(\lambda,\phi)$) is enumerated by the number of integral points in $\mathrm{Hive}(\mu, \lambda, \nu, \phi)$.
\end{proposition}

\subsection{Hives and GT patterns}
Given a hive $h = (h_{i,j})$, one can obtain a GT pattern $P(h):= p = (p_{i,j})$ by taking the successive differences along the rows (i.e., $p_{i,j} := h_{i+1,j+1} - h_{i+1,j}$). The assumption that the contents of the northeast and southeast rhombi of $h$ are non-negative translates to the GT inequalities in \eqref{eq_GT_inequality}. 
\medskip

One could also take the northeast differences and get a GT pattern $\hat{P}(h)$ out of it in a similar fashion (i.e., $\hat{p}_{i,j} := h_{m-i+j,m-i} - h_{m-i+j-1,m-i}$). In this case, the contents of the northeast and vertical rhombi of $h$ being non-negative translates to the GT inequalities in \eqref{eq_GT_inequality}

\begin{proposition}\label{prop_LR_symmetry}\cite[Proposition 4]{KRV}, \cite[Appendix A]{Buch}

    Let $h \in \mathrm{Hive}(\lambda, \mu, \nu)$ be an integral hive. Then,
    \begin{enumerate}
        \item $T(P(h))$ is a $\lambda$-dominant tableau of shape $\mu$ and weight $\nu - \lambda$. 
        \item  $C(\hat{P}(h))$ is a $\mu$-dominant contretableau of shape $\lambda$ and weight $\nu - \mu$.
    \end{enumerate}
    In fact, the map $T\circ P$ (resp. $C\circ \hat{P}$) is a bijection from $\mathrm{Hive}(\lambda, \mu, \nu)$ onto $LR_{\lambda,\mu} ^{\nu}$ (resp. the set of $\mu$-dominant contretableaux of shape $\lambda$ and weight $\nu-\lambda$).
\end{proposition}

\subsection{Proof of the main theorem}
Fix $m= 2n$ in what follows and let $\lambda, \mu$ and $\nu$ be partitions such that length of $\nu$ is at most $2n-1$, length of $\mu$ is at most $n$ and $\lambda$ is even. 
\begin{proposition}\label{prop_LRS_to_fl_tabs}
    The image of $LRS(\nu/\mu, \lambda)$ under the companion map $c$ is the set of $\mu$-dominant tableaux in $\mathrm{SSYT}(\lambda, \phi)$ of weight $\nu - \mu$ where $\phi = (\phi_1, \ldots, \phi_{2n})$ is the flag for which $\phi_k = n+ \lfloor k/2 \rfloor$.
\end{proposition}
\begin{proof}
    The image is a $\mu$-dominant tableau of shape $\lambda$ and weight $\nu - \mu$ is known already (see \S~\ref{sec:branching_models}). The Sundaram condition precisely translates under the map $c$ to the fact that given any positive integer $t$, the entries in rows $2t+1$ are bounded above by $n+t$. The remaining bounds follow from the fact that $c(T)$ is $\mu$-dominant and $\mu$ has at most $n$ parts.
\end{proof}

\begin{remark}\label{rem-hive-rest}
By Proposition~\ref{prop_fl_hive_to_tab}, we see that $\mathrm{LRS}(\nu/\mu,\lambda)$ is in one-to-one correspondence with $\mathrm{Hive(\mu, \lambda, \nu, \phi)}$ where $\phi$ is as in Proposition~\ref{prop_LRS_to_fl_tabs}. This gives a new combinatorial model for the branching coefficients $c^\nu _\lambda$, namely, the set of integral points in the (disjoint) union of the flagged hive polytopes $\bigcup \mathrm{Hive}(\mu,\lambda,\nu,\phi)$ where, the union is over all even partitions $\lambda$.
\end{remark}

\begin{proposition}
\label{S}
    Given any $h \in \mathrm{Hive}(\mu, \lambda, \nu, \phi)$, we have $\mathrm{rect}(C(\hat{P}(h))) \in \mathrm{LRK}^{\nu}_{\lambda,\mu}$. 
\end{proposition}

\begin{proof}
It is a well-known fact that the Knuth relations preserve $\mu$ dominance. Therefore, it follows from Proposition~\ref{prop_LR_symmetry} that $\mathrm{rect}(C(\hat{P}(h))) \in \mathrm{LR}^{\nu}_{\lambda,\mu}$. By Proposition~\ref{prop_SI} it is enough to prove that $T(\hat{P}(h))$ satisfies the Kwon property. 
\begin{figure}[h]
    \centering
    \begin{tikzpicture}
    \node at (0,0) {$\bullet$};
    \node at (1,0) {$\bullet$};
    
    \node at (0,-0.3) {$h_{i+1,j}$};
    \node at (1.6,-0.3) {$h_{i+1,j+1}$};
  
    \node at (0.5,0.5*1.732) {$\bullet$};
    \node at (1.5,0.5*1.732) {$\bullet$};
        
    \node at (0.5,0.7*1.732) {$h_{i,j}$};
    \node at (1.9,0.65*1.732) {$h_{i,j+1}$};
    \draw[ultra thick, draw=black, fill=violet, fill opacity=0.5] (0,0) -- (1,0) -- (1.5,0.5*1.732) -- (0.5,0.5*1.732) -- cycle;
    \end{tikzpicture}
    \caption{}
    \label{fig:one_flat_rhombi}
\end{figure}
Observe that a flat NE rhombus like the one in Figure~\ref{fig:one_flat_rhombi} gives rise to the equality $\hat{p}_{n-j,i-j} = \hat{p}_{n-j+1,i-j+1}$ because the content of the NE rhombus is $h_{i,j} - h_{i+1,j} - (h_{i,j+1} - h_{i+1,j+1}) = \hat{p}_{n-j,i-j} - \hat{p}_{n-j+1,i-j+1} = 0$). Also, we have $\hat{p}_{2n,n+t} = 0$ for all positive integers $t$ (because $\mu$ has atmost $n$ parts). These relations in $\hat{P}(h)$ translate exactly to the Kwon condition on $T(\hat{P}(h))$.
\begin{figure}[h]
    \centering
    \begin{center}
    \begin{tikzpicture}
        \node at (0,0) {$\bullet$};
        \node at (1,0) {$\bullet$};
        \node at (2,0) {$\bullet$};
        \node at (3,0) {$\bullet$};
		\node at (4,0) {$\bullet$};
        \node at (5,0) {$\bullet$};
        \node at (6,0) {$\bullet$};
        
        \node at (0.5,0.5*1.732) {$\bullet$};
        \node at (1.5,0.5*1.732) {$\bullet$};
        \node at (2.5,0.5*1.732) {$\bullet$};
        \node at (3.5,0.5*1.732) {$\bullet$};
        \node at (4.5,0.5*1.732) {$\bullet$};
        \node at (5.5,0.5*1.732) {$\bullet$};
        
        \node at (1,1.732) {$\bullet$};
        \node at (2,1.732) {$\bullet$};
        \node at (3,1.732) {$\bullet$};
        \node at (4,1.732) {$\bullet$};
        \node at (5,1.732) {$\bullet$};

        \node at (1.5,1.5*1.732) {$\bullet$};
        \node at (2.5,1.5*1.732) {$\bullet$};
        \node at (3.5,1.5*1.732) {$\bullet$};
        \node at (4.5,1.5*1.732) {$\bullet$};
        
        \node at (2,2*1.732) {$\bullet$};
        \node at (3,2*1.732) {$\bullet$};
        \node at (4,2*1.732) {$\bullet$};

        \node at (2.5,2.5*1.732) {$\bullet$};
        \node at (3.5,2.5*1.732) {$\bullet$};

        \node at (3,3*1.732) {$\bullet$};
        
        %horizontal lines
        \draw (0,0) -- (6,0);
		\draw (0.5, 0.5*1.732) -- (5.5, 0.5*1.732);
		\draw (1, 1.732) -- (5, 1.732);
		\draw (1.5, 1.5*1.732) -- (4.5, 1.5*1.732);
        \draw (2, 2*1.732) -- (4, 2*1.732);
		\draw (2.5, 2.5*1.732) -- (3.5, 2.5*1.732);

        %NE lines
		\draw (0,0) -- (3, 3*1.732);
		\draw (1,0) -- (3.5, 2.5*1.732);
		\draw (2,0) -- (4, 2*1.732);
		\draw (3,0) -- (4.5, 1.5* 1.732);
        \draw (4,0) -- (5, 1.732);
		\draw (5,0) -- (5.5, 0.5* 1.732);

        %NW lines
		\draw (1,0) -- (0.5, 0.5*1.732);
		\draw (2,0) -- (1,1.732);
		\draw (3,0) -- (1.5, 1.5*1.732);
        \draw (4,0) -- (2, 2*1.732);
        \draw (5,0) -- (2.5, 2.5*1.732);
        \draw (6,0) -- (3, 3*1.732);

        \draw [ultra thick, draw=black, fill=violet, fill opacity=0.5]
       (0,0) -- (1.5,1.5*1.732) -- (2.5,1.5*1.732) -- (2,1.732) -- (4,1.732) -- (3.5,0.5*1.732) -- (5.5,0.5*1.732) -- (5,0) -- cycle;

       \draw [ultra thick, draw=black]
       (0.5,0.5*1.732) -- (3.5,0.5*1.732);
       \draw [ultra thick, draw=black]
       (1,1.732) -- (2,1.732);
        \draw [ultra thick, draw=black]
       (2,1.732) -- (1,0);
       \draw [ultra thick, draw=black]
       (3,1.732) -- (2,0);
       \draw [ultra thick, draw=black]
       (3.5,0.5*1.732) -- (3,0);
       \draw [ultra thick, draw=black]
       (4.5,0.5*1.732) -- (4,0);
	\end{tikzpicture}
 \begin{tikzpicture}
    \node at (-0.5,-0.5*1.732) {$\hat{p}_{6,1}$};
    \node at (0.5,-0.5*1.732) {$\hat{p}_{6,2}$};
    \node at (1.5,-0.5*1.732) {$\hat{p}_{6,3}$};
    \node at (2.5,-0.5*1.732) {$\hat{p}_{6,4}$};
    \node at (3.5,-0.5*1.732) {$\hat{p}_{6,5}$};
    \node at (3,-1.732) {$0$};
    \node at (4,-1.732) {$0$};
    \node at (4.5,-0.5*1.732) {$\hat{p}_{6,6}$};
    \node at (5,-1*1.732) {$0$};

    \node at (0,0) {$\hat{p}_{5,1}$};
    \node at (1,0) {$\hat{p}_{5,2}$};
    \node at (2,0) {$\hat{p}_{5,3}$};
    \node at (3,0) {$\hat{p}_{5,4}$};
    \node at (4,0) {$\hat{p}_{5,5}$};
 
    \node at (0.5,0.5*1.732) {$\hat{p}_{4,1}$};
    \node at (1.5,0.5*1.732) {$\hat{p}_{4,2}$};
    \node at (2.5,0.5*1.732) {$\hat{p}_{4,3}$};
    \node at (3.5,0.5*1.732) {$\hat{p}_{4,4}$};
    
    \node at (1,1.732) {$\hat{p}_{3,1}$};
    \node at (2,1.732) {$\hat{p}_{3,2}$};
    \node at (3,1.732) {$\hat{p}_{3,3}$};
    
    \node at (1.5,1.5*1.732) {$\hat{p}_{2,1}$};
    \node at (2.5,1.5*1.732) {$\hat{p}_{2,2}$};
        
    \node at (2,2*1.732) {$\hat{p}_{1,1}$};
    
    \draw [double] (2.2,1.8*1.732) -- (2.3,1.7*1.732);
    \draw [double] (2.7,1.3*1.732) -- (2.8,1.2*1.732);
    \draw [double] (3.2,0.8*1.732) -- (3.3,0.7*1.732);
    \draw [double] (3.7,0.3*1.732) -- (3.8,0.2*1.732);
    \draw [double] (4.2,-0.2*1.732) -- (4.3,-0.3*1.732);
    \draw [double] (4.7,-0.7*1.732) -- (4.8,-0.8*1.732);

    \draw [double] (2.2,0.8*1.732) -- (2.3,0.7*1.732);
    \draw [double] (2.7,0.3*1.732) -- (2.8,0.2*1.732);
    \draw [double] (3.2,-0.2*1.732) -- (3.3,-0.3*1.732);
    \draw [double] (3.7,-0.7*1.732) -- (3.8,-0.8*1.732);

    \draw [double] (2.2,-0.2*1.732) -- (2.3,-0.3*1.732);
    \draw [double] (2.7,-0.7*1.732) -- (2.8,-0.8*1.732);
\end{tikzpicture}
\end{center}
\caption{A flagged hive $h$ and the corresponding GT pattern $\hat{P}(h)$.}
\label{fig_fhives_to_NE-GT}
\end{figure}
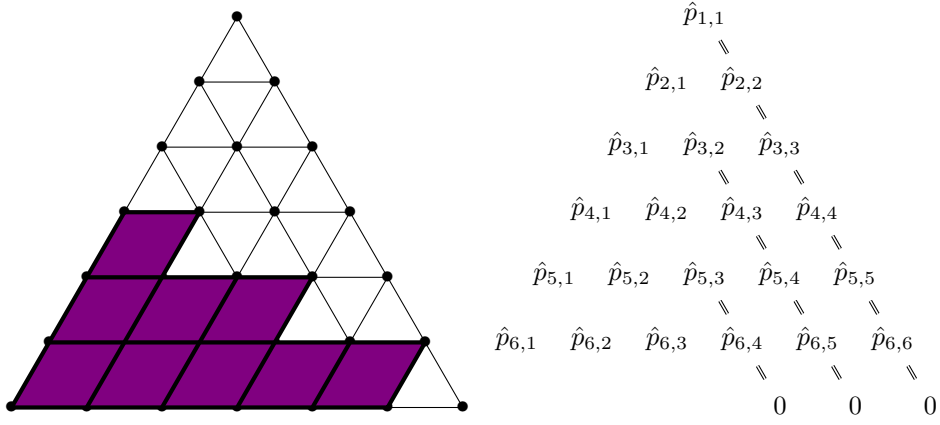

\end{proof}

We now restate and prove Theorem~\ref{maintheorem}:
\begin{theorem}
\label{final}
The following map is a bijection:
    \begin{equation}
      \op{rect} 
  \circ \ C 
 \circ \hat{P} \circ \varphi \circ c: \ \mathrm{LRS}(\nu/\mu, \lambda) \longrightarrow{} \mathrm{LRK}^{\nu}_{\lambda,\mu}.
    \end{equation}
\end{theorem}

\begin{proof}
Since all the maps involved here are invertible, it is enough to show that the image is a subset of $\mathrm{LRK}^\nu _{\lambda, \mu}$. By Proposition \ref{prop_LRS_to_fl_tabs} the companion map $c$ induces a bijection between $\op{LRS}(\nu/\mu,\lambda)$ and $\op{SSYT}(\lambda, \phi) \cap \op{LR}^{\nu}_{\mu,\lambda}$. By Proposition \ref{prop_fl_hive_to_tab}, the map $\varphi$ induces a bijection between $\op{SSYT}(\lambda, \phi) \cap \op{LR}^{\nu}_{\mu,\lambda}$ and $\mathrm{Hive}(\mu, \lambda, \nu, \phi)$. Finally apply Proposition \ref{S}. 

\end{proof}
\newpage
\noindent
We work out our bijection for an example in detail: 
\begin{example}
Let $n=3$ and \[T = \Skew(0:0, 0,1,1,1|0:0, 1,2,2|0: 0, 2,3|0: 2,3,4) \in \operatorname{LRS}(\nu/\mu,\lambda)\] where $\nu = (5,4,3,3), \mu = (2,1,1), \hbox{ and }\lambda = (4,4,2,1)$. 

Then the companion tableau $c(T)$ is given by
\[c(T) = \Skew(0: 1,1,1,2|0: 2,2,3,4|0: 3,4|0: 4) \in \op{LR}^{\nu}_{\mu,\lambda}\]

The corresponding hive $h: = \varphi(c(T))$ and GT pattern $\hat{P}(h)$ turn out to be, respectively:
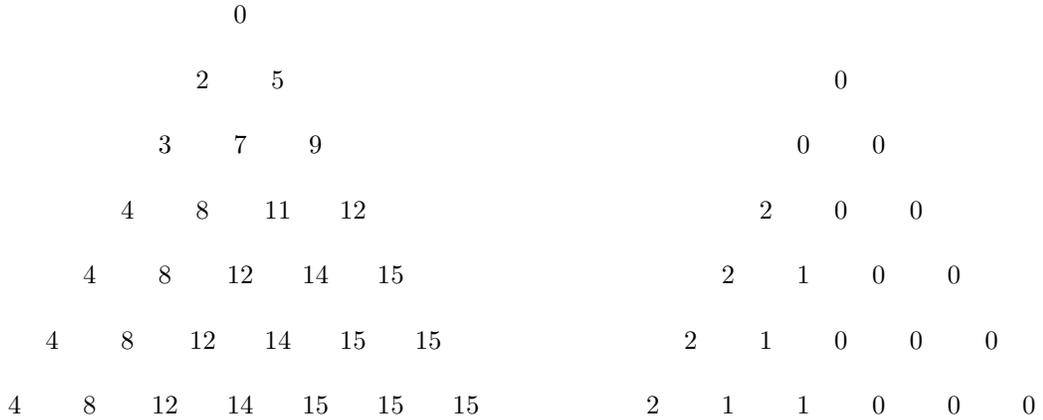
\begin{figure}[h]
    \centering
    \begin{tikzpicture}
    \node at (5,-1.732) {$4$};
    \node at (6,-1.732) {$8$};
    \node at (7,-1.732) {$12$};
    \node at (8,-1.732) {$14$};
    \node at (9,-1.732) {$15$};
    \node at (10,-1.732) {$15$};
    \node at (11,-1.732) {$15$};

    \node at (5.5,-0.5*1.732) {$4$};
    \node at (6.5,-0.5*1.732) {$8$};
    \node at (7.5,-0.5*1.732) {$12$};
    \node at (8.5,-0.5*1.732) {$14$};
    \node at (9.5,-0.5*1.732) {$15$};
    \node at (10.5, -0.5*1.732) {$15$};

    \node at (6,0) {$4$};
    \node at (7,0) {$8$};
    \node at (8,0) {$12$};
    \node at (9,0) {$14$};
    \node at (10,0) {$15$};
    
    \node at (6.5,0.5*1.732) {$4$};
    \node at (7.5,0.5*1.732) {$8$};
    \node at (8.5,0.5*1.732) {$11$};
    \node at (9.5,0.5*1.732) {$12$};

    \node at (7,1.732) {$3$};
    \node at (8,1.732) {$7$};
    \node at (9,1.732) {$9$ };

    \node at (7.5,1.5*1.732) {$2$};
    \node at (8.5,1.5*1.732) {$5$};
        
    \node at (8,2*1.732) {$0$};
\end{tikzpicture}~\begin{tikzpicture}

    \node at (3.5,0) {$\null$};
    
    \node at (5.5,-0.5*1.732) {$2$};
    \node at (6.5,-0.5*1.732) {$1$};
    \node at (7.5,-0.5*1.732) {$1$};
    \node at (8.5,-0.5*1.732) {$0$};
    \node at (9.5,-0.5*1.732) {$0$};
    \node at (10.5, -0.5*1.732) {$0$};

    \node at (6,0) {$2$};
    \node at (7,0) {$1$};
    \node at (8,0) {$0$};
    \node at (9,0) {$0$};
    \node at (10,0) {$0$};
    
    \node at (6.5,0.5*1.732) {$2$};
    \node at (7.5,0.5*1.732) {$1$};
    \node at (8.5,0.5*1.732) {$0$};
    \node at (9.5,0.5*1.732) {$0$};

    \node at (7,1.732) {$2$};
    \node at (8,1.732) {$0$};
    \node at (9,1.732) {$0$ };

    \node at (7.5,1.5*1.732) {$0$};
    \node at (8.5,1.5*1.732) {$0$};
        
    \node at (8,2*1.732) {$0$};
\end{tikzpicture}
\end{figure}

The Young tableau and the contretableau corresponding to $\hat{P}(h)$ are 
\[ T(\hat{P}(h)) = \Skew(0: 3,3 | 0: 4 |0: 6) \hbox{ and } C(\hat{P}(h)) = \Skew(0:0, 1 | 0:0, 3 |0: 4,4).\]

Finally, the rectification of the contretableau $C(\hat{P}(h))$ is \[\operatorname{rect}(C(\hat{P}(h)))) = \Skew(0: 1,4 | 0: 3 |0: 4) \in LRK^{\nu}_{\lambda,\mu}\]
\end{example}

\section{Acknowledgements}

J.T. was supported by the grant SONATA NCN UMO-2021/43/D/ST1/02290 and partially supported by the grant MAESTRO NCN-UMO-2019/34/A/ST1/00263.

\bibliographystyle{alpha}

%\bibliography{bibliography}

\end{document}